\documentclass[a4paper,11pt,reqno]{amsart}
\usepackage{color}
\usepackage{enumerate}
\usepackage{amssymb, amsmath}

\theoremstyle{plain}
\newtheorem*{theorem*}{Theorem}
\newtheorem{corollary}{Corollary}
\newtheorem{lemma}{Lemma}

\newtheorem*{proposition*}{Proposition}

\newtheorem*{definition*}{Definition}

\theoremstyle{remark}

\newtheorem*{claim*}{Claim}
\newtheorem*{remark*}{Remark}
\newtheorem*{example*}{Example}
\newtheorem*{notation*}{Notation}

\def\N{{\mathbb N}}

\def\R{{\mathbb R}}

\def\Q{{\mathbb Q}}

\newcommand{\E}{\mathcal{E}}

\begin{document}
\title[Gradient Flows for semiconvex Functions on Metric Measure Spaces]{Gradient Flows for semiconvex Functions on Metric Measure Spaces -- Existence, Uniqueness and Lipschitz Continuity}
\author{Karl-Theodor Sturm}
 \address{
 University of Bonn\\
 Institute for Applied Mathematics\\
 Endenicher Allee 60\\
 53115 Bonn\\
 Germany}
  \email{sturm@uni-bonn.de}

\thanks{
The author gratefully acknowledges  support by the European Union through the
ERC Advanced Grant ``Metric measure spaces and Ricci curvature - analytic, geometric, and probabilistic challenges" (``RicciBounds'')
as well as support by the German Research Foundation through the Hausdorff Center for Mathematics and the Collaborative Research Center 1060 ``The Mathematics of Emergent Effects''.
}

\maketitle

\begin{abstract} 
Given any continuous, lower bounded and $\kappa$-convex function $V$ on a metric measure space $(X,d,m)$ which is infinitesimally Hilbertian and satisfies some synthetic lower bound for the Ricci curvature in the sense of Lott-Sturm-Villani, we prove existence and uniqueness for the (downward) gradient flow for $V$. Moreover, we prove  Lipschitz continuity of the flow w.r.t. the starting point 
\begin{equation*}
d(x_t,x'_t)\le e^{-\kappa\, t} d(x_0,x_0').
\end{equation*}
\end{abstract}


Throughout this paper, let
 $(X,d)$  be a locally compact geodesic space and let $m$ be a locally finite Borel measure with full topological support. We always assume that  the
metric measure space  $(X,d,m)$  satisfies the RCD$(K,\infty)$-condition for some  finite number $K\in\R$.
Recall that this means that $\mathrm{Ent}(.|m)$, the Boltzmann entropy w.r.t. $m$, is weakly 
$K$-convex  
on the geodesic spaces $(\mathcal{P}_2(X),W_2)$  and that $\E$, the Cheeger energy on $(X,d,m)$, is a quadratic functional on $L^2(X,m)$, \cite{St1}, \cite{LV1}, \cite{AGSb}, \cite{EKS}.
Among many others, it implies that
$\int_X e^{- C'\, d^2(x,x')}\,m(dx)<\infty$ for some $C'\in\R, x'\in X$.
Note that the latter in turn implies 
$\mathrm{Ent}(\mu|m)\ge -C''-C'\, W_2^2(\mu,\delta_{x'})$ for each $\mu\in \mathcal{P}_2(X)$. 

Here and henceforth $W_2$ will denote the $2$-Kantorovich-Wasserstein distance and $\mathcal{P}_2(X)$ the $2$-Wasserstein space consisting of those Borel probability measures $\mu$ on $X$ which satisfy $W_2(\mu,\delta_{x'})<\infty$ (for some, hence all, $x'\in X$). Recall that
$  \mathrm{Ent}(\mu|m)= \int \rho\log \rho \, dm$
provided $\mu=\rho m$ is absolutely continuous w.r.t.~$m$ and
$\int\rho(\log\rho)_+\,  dm<\infty$. Otherwise
$\mathrm{Ent}(\mu|m)=+\infty$.
Also recall that $\E(u)=\liminf\big\{\int_X |Du_n|^2\,dm: \ u_n\in \mathrm{Lip}(X), u_n\to u\ \mathrm{ in } \ L^2(X,m)\big\}$ with
$|Du|(x)=\limsup_{y\to x}\frac{|u(x)-u(y)|}{d(x,y)}$.

Moreover, let a function
$V: X\to (-\infty,+\infty]$ be given which is  continuous (w.r.t. the topology of the extended real line) and
bounded from below by $-C_0-C_1 d^2(.,x_0)$ for some $C_0,C_1\in\R, x_0\in X$.
Put $X_0=\{V<\infty\}$ and
$X'=\overline{X_0}$.

\begin{definition*}
\begin{itemize}
\item[(i)]
Given a number $\kappa\in\R$, we say that 
 $V$ is weakly $\kappa$-convex  if for each $x_0,x_1\in X$ there exists a geodesic  $\gamma:[0,1]\to X$ connecting them (i.e. $\gamma_0=x_0, \gamma_1=x_1$) such that
\begin{equation}\label{conv}
V(\gamma_t)\le (1-t) V(\gamma_0)+t V(\gamma_1)-\kappa t(1-t) |\dot\gamma|^2/2
\end{equation}
for each $t\in [0,1]$.
If \eqref{conv} holds for all geodesics $\gamma$ in $X$ and all $t\in[0,1]$ then $V$ is called strongly $\kappa$-convex.

\item[(ii)] Given $x_0\in X$,  an EVI$_\kappa$-gradient flow for $V$ starting in $x_0$ is  an absolutely continuous curve $(x_t)_{t\ge0}$ in $X$ emanating in $x_0$ such that for all $z\in X_0$ and a.e. $t>0$
\begin{equation}\label{evi}
\frac{d}{dt} \frac12 d^2(x_t,z) +\kappa \, d^2(x_t,z)\le V(z)-V(x_t).
\end{equation}
\end{itemize}
\end{definition*}
Note that here and in the sequel -- as common in metric geometry -- a geodesic is always minimizing and parametrized by arc length.

The main result of this paper is

\begin{theorem*} If $V$ is weakly $\kappa$-convex 
then for each $x_0\in X'$ there exists a unique EVI$_\kappa$-gradient flow for $V$ starting in ${x_0}$. 
\end{theorem*}

It is well-known -- and easy to see -- that the existence of an EVI$_\kappa$-gradient flow for $V$ starting in $x_0$ implies its uniqueness and the `contraction' property (or `expansion' bound)
\begin{equation}\label{contract}
d(x_t,x'_t)\le e^{-\kappa\, t} d(x_0,x_0')
\end{equation}
for each pair of points $x_0,x_0'\in X'$ and associated flows  $(x_t)_{t\ge0}$ and $(x_t')_{t\ge0}$. Moreover, every EVI$_\kappa$-gradient flow is also a gradient flow in the sense of energy dissipation
\begin{equation}\label{dissi}
V(x_t)=V(x_s)-\frac12\int_s^t \left(|\dot x_r|^2+|\nabla^- V|^2(x_r)\right)dr
\end{equation}
for all $0\le s\le t$ where as usual
$|\nabla^- V|(x):=\limsup_{y\to x}\frac{[V(x)-V(y)]_+}{d(x,y)}$.
  In particular, $t\mapsto V(x_t)$ is non-increasing.

\begin{corollary}
For each $x_0\in X'$ there exists a unique curve $(x_t)_{t\ge0}$ such that 
$\mu_t=\delta_{x_t}$ is the EVI$_\kappa$-gradient flow for $S$ starting in $\mu_0=\delta_{x_0}$. 

The curve $(x_t)_{t\ge0}$ is the unique EVI$_\kappa$-gradient flow for $V$ starting in ${x_0}$.
\end{corollary}

\begin{proof} 
Put $\mu_0=\delta_{x_0}$. Then $\mathrm{Var}_2(\mu_0)=0$ and thus, according to the previous lemma, also $\mathrm{Var}_2(\mu_t)=0$. This implies that each of the measures $\mu_t$ is supported by a single point, say $x_t$.

Moreover, the fact that
 the  curve $(\delta_{x_t})_{t\ge0}$ in $\mathcal{P}_2(X')$ is the EVI$_\kappa$-gradient flow for $S$ starting in $\mu_0=\delta_{x_0}$ immediately implies that the curve $(x_t)_{t\ge0}$ in $X'$ is the (unique) EVI$_\kappa$-gradient flow for $V$ starting in ${x_0}$. (To verify \eqref{evi},
just apply \eqref{evi-disc} to Diracs $\nu=\delta_{z}$.)
\end{proof}

\begin{corollary} The following properties are equivalent:
\begin{itemize}
\item[(i)]
 $V$ is weakly $\kappa$-convex;
\item[(ii)]
 $V$ is strongly $\kappa$-convex; 
\item[(iii)]
 $V$ is  $\kappa$-convex in EVI-sense, i.e. for each $x_0\in X'$ there exists a curve $(x_t)_{t\ge0}$ in $X'$ such that for all $z\in X_0$ and all $t>0$
\begin{equation*}
\frac{d}{dt} \frac12  d^2(x_t,z) +\kappa d^2(x_t,z)\le V(z)-V(x_t).
\end{equation*}
\end{itemize}
\end{corollary}

Note that the implications (iii)$\Rightarrow$(ii)$\Rightarrow$(i)  hold true in any geodesic space, see \cite{EKS}.

\medskip

\begin{remark*}
Our main result extends previous results in various respects.

In the setting of Alexandrov spaces (i.e. metric spaces with synthetic lower bounds for the sectional curvature), the existence, uniqueness and Lipschitz continuity of gradient flows for semiconvex functions has been derived in an unpublished manuscript by  Perelman and Petrunin \cite{Pet}. The proof has been worked out and published by  Plaut \cite{Pla} and  Lytchak \cite{Lyt}.
This result requires no local compactness of the state space. The function
$V$, however, is always required to be bounded and Lipschitz continuous.

In the setting of Hilbert spaces (and to some extent also for metric spaces), a far reaching theory of gradient flows for lower semicontinuous functions (with suitable lower growth bounds and compact sublevel sets)  has been developed by Brezis as well as by De Giorgi and his school, see in particular the monograph \cite{AGS-book} by Ambrosio-Gigli-Savar\'e.
In this context, however, the state space $X$ has to be flat (or of nonpositive curvature).
 
In the setting of metric measure spaces with synthetic Ricci bounds, currently no  direct proof -- without passing to the Wasserstein space but instead using more detailed properties of the Hessian of $V$, cf. \cite{St3} -- is known to the author.

\end{remark*}

\bigskip
\medskip

The Theorem will  be proven according to the following line of argumentation in the subsequent  lemmata:
\begin{itemize}
\item There exists an EVI$_{\kappa_n}$-gradient flow for the functional $S_n(\mu)=\frac1n \mathrm{Ent}(\mu|m)+\int Vd\mu$
on  $\mathcal{P}_2(X')$  with $\kappa_n=\frac1n K+\kappa$.
\item In the limit $n\to\infty$ this yields an EVI$_\kappa$-gradient flow for the functional $S(\mu)=\int Vd\mu$
on $\mathcal{P}_2(X')$.
\item The EVI$_\kappa$-gradient flow for $S$ is non-diffusive. In particular, the flow starting in a Dirac mass will not spread out: it is  a Dirac mass in a moving point.
\end{itemize}

\begin{lemma}
 There exists an EVI$_{\kappa_n}$-gradient flow for the functional $S_n(\mu)=\frac1n \mathrm{Ent}(\mu|m)+\int Vd\mu$
on $\mathcal{P}_2(X')$ with $\kappa_n=\frac1n K+\kappa$.
\end{lemma}

\begin{proof} 
{\bf (i)} Let us begin with some general remarks. Since $X$ is geodesic and locally compact, it is 'proper' in the sense that all closed balls are compact. By means of Prohorov's theorem this  implies that $W_2$-bounded subsets of $\mathcal{P}_2(X)$ are relatively compact (w.r.t. the usual topology of weak convergence of measures).

{\bf (ii)} As a first step, we replace the space $X$  by $X'$ and the measure $m$ by $m':=1_{X_0}m$.
Weak $\kappa$-convexity of the function $V$ implies weak convexity of the set $X_0$, i.e., for each pair of point $x_0,x_1\in X_0$ there exists a connecting geodesic which stays in $X_0$ (and minimizes the distance in $X$).

Together with the fact that RCD$(K,\infty)$ implies strong $K$-convexity of the Boltzmann entropy $\mathrm{Ent}(.|m)$ on  $\mathcal{P}_2(X)$ 
(\cite{AGSb}, Prop. 2.23), 
it yields
weak $K$-convexity of the Boltzmann entropy $\mathrm{Ent}(.|m')$ on  $\mathcal{P}_2(X')$.
Moreover, the Cheeger energy for $(X',d,m')$ is still quadratic. (The relaxed gradients w.r.t. $m$ and $m'$, resp., are local operators,  they coincide on $X_0$.) 
Thus $(X',d,m')$ satisfies the condition RCD$(K,\infty)$, cf. \cite{AGSb}, Thm. 4.19 (iii).

{\bf (iii)} Next we consider the weighted mms $(X',d, m_1')$ with $m_1'=e^{-V}m'$. 
Since $V$ is continuous and bounded from below by 
 $-C_0-C_1 d^2(.,x_0)$, it follows that the functional
$S: \mu\mapsto \int V\,d\mu$
on $\mathcal{P}_2(X')$ is lower semicontinuous and bounded from below by
$-C_0-C_1 W_2^2(.,\delta_{x_0})$.
Indeed, if $\mu_n\to\mu$ in  $\mathcal{P}_2(X')$ then $W_2(\mu_n,\delta_{x_0})\to W_2(\mu,\delta_{x_0})$ and
$$\liminf_{n\to\infty}\int \tilde Vd\mu_n\ge \liminf_{n\to\infty}\int \tilde V_id\mu_n=\int\tilde V_id\mu$$
for each $i\in\N$ where
$\tilde V:=V+C_1\cdot d^2(.,x_0)$ and $\tilde V_i:=\tilde V\wedge i$ (which is bounded and continuous).
Therefore, $\liminf_{n\to\infty}\int \tilde Vd\mu_n\ge \int\tilde Vd\mu$ and thus in turn $\liminf_{n\to\infty}\int  Vd\mu_n\ge \int Vd\mu$.

The weak $\kappa$-convexity of $V$ on $X$ implies that $S$ is weakly $\kappa$-convex on 
 $\mathcal{P}_2(X')$. 
Indeed, if $\mu_0=\sum_{i=1}^\infty \lambda_i\cdot \delta_{x_i}$ and
$\mu_1=\sum_{i=1}^\infty \lambda_i\cdot \delta_{y_i}$
for some $\lambda_i\ge0$ and $x_i,y_i\in X'$ 
are given in such a way that $q=\sum_{i=1}^\infty \lambda_i\cdot \delta_{(x_i,y_i)}$ is an optimal coupling of them
and if
$z_i$ is an $t$-intermediate point of $x_i, y_i$ (i.e. $\frac1t d(x_i,z_i)=\frac1{1-t}d(y_i,z_i)=d(x_i,y_i)$)
with $V(z_i)\le (1-t)V(x_i)+t V(y_i)-\frac\kappa2 t(1-t)d^2(x_i,y_i)$ 
for each $i\in\N$ then
$\mu_t:=\sum_{i=1}^\infty \lambda_i\cdot \delta_{z_i}$ is an $t$-intermediate point of $\mu_0, \mu_1$ and
 $$S(\mu_t)\le (1-t)S(\mu_0)+t S(\mu_1)-\frac\kappa2 t(1-t)W_2^2(\mu_0,\mu_1).$$ 
For arbitrary discrete measures $\mu_0=\sum_{i=1}^\infty \alpha_i\cdot \delta_{x_i}$ and $\mu_1=\sum_{j=1}^\infty \beta_j\cdot \delta_{y_j}$, each optimal coupling of them can be written as
$q=\sum_{i,j=1}^\infty \lambda_{ij}\cdot \delta_{(x_i,y_j)}$ for suitable $\lambda_{ij}\ge0$ with
$\sum_{k=1}^\infty \lambda_{ik}=\alpha_i$ and $\sum_{k=1}^\infty \lambda_{kj}=\beta_j$. Replacing $\alpha_i\delta_{x_i}$ and $\beta_j\delta_{y_j}$ by
$\sum_{k=1}^\infty \lambda_{ik}\delta_{x_i}$ and $\sum_{k=1}^\infty \lambda_{kj}\delta_{y_j}$, resp., and re-labeling the indices, we may assume without restriction that $\mu_0, \mu_1$ and $q$ are always given as above. 

Given general  $\mu_0,\mu_1\in\mathcal{P}_2(X')$, choose $\mu^n_s \in\mathcal{P}_2(X')$ for $s\in\{0,1\}$ of the above type with $\mu^n_s\to\mu_s$ and $\int V d\mu^n_s\to \int V d\mu_s$ as $n\to\infty$. One way to achieve this, is to fix for each $n\in\N$ a partition $X'=\bigcup_i X_i^n$ by sets of diameter less than $1/n$, to select points $x_i^n\in \overline X_i^n$ with $V(x_i^n)=\inf_{y\in \overline X_i^n}V(y)$, and to put $\mu_s^n=\sum_i \mu_s(X_i^n)\delta_{x_i^n}$. Then obviously $W_2(\mu_s^n,\mu_s)\le1/n$ and
$\int V d\mu^n_s\le \int V d\mu_s$. Thus the claim follows by lower semicontinuity of $\int Vd\mu$.

Then for given $t\in(0,1)$ choose $t$-intermediate points $\mu^n_t$ of $\mu^n_0,\mu^n_1$ satisfying the above inequality for $\kappa$-convexity.
Properness of $X$ and $W^2_2$-boundedness of $(\mu^n_t)_{n\in\N}$
implies that for each $t\in(0,1)$ there exists a converging subsequence, again denoted by $\mu^n_t, n\in\N$, of $t$-intermediate points of $\mu^n_0,\mu^n_1$ with limit point $\mu_t$ which then must be an $t$-intermediate point of $\mu_0,\mu_1$. Lower semicontinuity of $S$ implies
\begin{eqnarray*}
S(\mu_t)&\le&
\liminf_{n\to\infty}S(\mu^n_t)\\
&\le&
\liminf_{n\to\infty}\left[
(1-t)S(\mu^n_0)+t S(\mu^n_1)-\frac\kappa2t(1-t)W_2^2(\mu^n_0,\mu^n_1)\right]\\
&=&\left[
(1-t)S(\mu_0)+t S(\mu_1)-\frac\kappa2 t(1-t)W_2^2(\mu_0,\mu_1)\right].
\end{eqnarray*}
This is the weak $\kappa$-convexity of $S$ on  $\mathcal{P}_2(X')$.

Since $\mathrm{Ent}(\mu|m')+\int Vd\mu=\mathrm{Ent}(\mu|e^{-V}m')$
and since $\mathrm{Ent}(\mu|m')$ is strongly $K$-convex on $\mathcal{P}_2(X')$ this also yields the weak $(K+\kappa)$-convexity of $\mathrm{Ent}(\mu|e^{-V}m')$ on $\mathcal{P}_2(X')$ or, in other words,
the CD$(K+\kappa,\infty)$-condition for $(X',d,m_1')$.

{\bf (iv)} The next step will be to prove that the Cheeger energy for  $(X',d,m_1')$ is quadratic or, equivalently, that the relaxed gradients in $(X',d,m_1')$ are quadratic. The assumption on the quadratic lower bound for $V$ guarantees the requested integrability condition for $m_1'$. Moreover, the local boundedness of $V$ guarantees that on each closed ball in $X_0$, the measures $m'$ and $m_1'$ are equivalent with Radon-Nikodym density bounded from above and below (away from 0). Thus according to \cite{AGSa}, Lemma 4.11, the relaxed gradients w.r.t. $m'$ and w.r.t. $m_1'$ coincide. Since the former already was proven to be quadratic this finally proves the  RCD$(K+\kappa, \infty)$-condition for $(X',d,m_1')$.

{\bf (v)}
These observations will now be applied with $n\cdot V$ in the place of $V$ which then yields that the space
$(X',d,e^{-nV}m')$  satisfies the RCD$(\tilde\kappa_n,\infty)$-condition  with $\tilde\kappa_n= K+n\kappa$, \cite{AGSb}, Prop. 6.19. The latter in turn implies the existence of an 
EVI$_{\tilde\kappa_n}$-gradient flow for the functional $\tilde S_n(\mu)= \mathrm{Ent}(\mu|e^{-nV}m')= \mathrm{Ent}(\mu|m')+n\int Vd\mu$
on $\mathcal{P}_2(X')$.
Replacing $\tilde S_n$ by $S_n=\frac1n \tilde S_n$ and rescaling time by the factor $n$ yields the claim. 
\end{proof}

\begin{lemma}
 There exists an EVI$_{\kappa}$-gradient flow for the functional $S(\mu)=\int Vd\mu$
on $\mathcal{P}_2(X')$.
\end{lemma}

\begin{proof} 

{\bf (i)}
Let $(\mu_t^n)_{t>0}$ denote  the EVI$_{\kappa_n}$-gradient flow for the functional $S_n(\mu)=\frac1n \mathrm{Ent}(\mu|m)+\int Vd\mu$
on $\mathcal{P}_2(X')$ with $\kappa_n=\frac1n K+\kappa$. It is unique and well-defined for each starting point $\mu_0\in\mathcal{P}_2(X')$. 
The property (\ref{evi}) ($\forall t>0$) -- with $S_n$ and $W_2$ in the place of $V$ and $d$ -- is obviously equivalent to the property
\begin{equation}\label{evi2}
\frac12  W^2_2(\mu_t^n,\nu)e^{2\kappa_n t}-\frac12  W^2_2(\mu^n_s,\nu)e^{2\kappa_n s}\le 
\frac1{2\kappa_n}(e^{2\kappa_n t}-e^{2\kappa_n s})[
 S_n(\nu)- S_n(\mu_t^n)]
\end{equation}
($\forall 0\le s<t$).
In the case $\kappa=0$, the prefactor $\frac1{2\kappa}(e^{2\kappa t}-e^{2\kappa s})$ has to replaced by $(t-s)$.

{\bf (ii)} Note that our assumptions on $m$ and $V$ imply that
$S_n(\nu)\le S_1(\nu)+C''+C'\, W_2^2(\nu,\delta_{x'})\le C^*$
whenever $S_1(\nu)<\infty$
and that
$-S_n(\mu_t^n)\le C''+C'\, W_2^2(\mu_t^n,\delta_{x'})+C_0+C_1\, W_2^2(\mu_t^n,\delta_{x_0})$.
The differential version of \eqref{evi2} (or, in other words, \eqref{evi} with $S_n$ and $W_2$) allows to estimate uniformly in $n$
\begin{eqnarray*}
\frac{d}{dt} \frac12 W_2^2(\mu^n_t,\nu) &\le& -\kappa \, W_2^2(\mu^n_t,\nu)+ S_n(\nu)-S_n(\mu^n_t)\\
&\le&  C + c\cdot\frac 12 W_2^2(\mu^n_t,\nu)
\end{eqnarray*}
which leads to the uniform growth bound
\begin{eqnarray*}
 \frac12 W_2^2(\mu^n_t,\nu) &\le& e^{ct} \cdot\frac12 W_2^2(\mu_0,\nu)+(e^{ct}-1)\cdot C/c
\end{eqnarray*}
with constants $C$ and $c$
 which may depend on $\nu$ but not on $n$.
Thus
for each $t\in(0,1)$, the sequence  $(\mu^n_t)_{n\in\N}$ is $W_2$-bounded.
 Local compactness of  $W_2$-bounded subsets of $\mathcal{P}_2(X')$ implies that there exists some limit point $\mu_t$ such that
 $\mu_t^n\to \mu_t$ weakly  -- at least after passing to a subsequence. A diagonal argument allows to choose this subsequence in such a way that convergence holds for all rational $t\ge0$.

{\bf (iii)}
The quadratic lower bound (in terms of $W_2$) for $\mu\mapsto \mathrm{Ent}(\mu|m)$ and the lower semicontinuity $\mu\mapsto\int Vd\mu$ imply
that 
$\liminf_{n\to\infty} S_n(\mu_t^n) \ge S(\mu_t)$.
Moreover, $\lim_{n\to\infty}S_n(\nu)=S(\nu)$ provided $\mathrm{Ent}(\nu)<\infty$.
Thus for all such $\nu$ and for all rational $0\le s< t$
\begin{eqnarray}\label{s-evi}
\frac12  W^2_2(\mu_t,\nu)e^{2\kappa t}-\frac12  W^2_2(\mu_s,\nu)e^{2\kappa s}\le 
\frac1{2\kappa}(e^{2\kappa t}-e^{2\kappa s})[
 S(\nu)- S(\mu_t)].
\end{eqnarray} 
Applying this estimate with $\nu=\mu_s$ we obtain  local Lipschitz continuity in $t\in (0,\infty)\cap\Q$ of the flow $(\mu_t)_t$ which allows to extend it to all $t$, still satisfying the previous estimate.
Thus, in particular,
\begin{equation}\label{evi-disc}
\frac{d}{dt} \frac12 W_2^2(\mu_t,\nu) +\kappa \, W_2^2(\mu_t,\nu)\le S(\nu)-S(\mu_t)
\end{equation}
for a.e.\ $t\ge0$.

{\bf (iv)}
The next step is to extend \eqref{s-evi} (and in turn therefore also \eqref{evi-disc})  to all $\nu\in\mathcal{P}_2(X')$,
i.e. to get rid of the constraint $\mathrm{Ent}(\nu|m)<\infty$.
That is, given  $\nu\in\mathcal{P}_2(X')$ we have to find a sequence of $\nu_n\in\mathcal{P}_2(X')$
with $\mathrm{Ent}(\nu_n|m')<\infty$ such that $W_2(\nu_n,\nu)\to0$ and $\int Vd\nu_n\to\int Vd\nu$ as $n\to\infty$. 

One easily verifies that it suffices to prove this approximation result for functions $V$ which are uniformly bounded from below. Indeed, assume that for each $i\in\N$ there exists a sequence of measures 
$\nu_{i,n}\in\mathcal{P}_2(X')$
with $\mathrm{Ent}(\nu_{i,n}|m')<\infty$ such that $W_2(\nu_{i,n},\nu)\le\frac1n$ and $\int V_id\nu_{i,n}\le\int V_id\nu+\frac1n$ for all $n$ where $V_i:=\max\{V,-i\}$. Then the sequence $\nu_{n,n}$ will do the job for $V$:
\begin{eqnarray*}
\liminf_{n\to\infty}
\int Vd\nu_{n,n}
&\le&
\liminf_{n\to\infty}\left[\int V_n d\nu_{n,n}\right]
\le
\liminf_{n\to\infty}\left[\int V_n d\nu+\frac1n\right]=\int Vd\nu.
\end{eqnarray*}
Similarly, a simple monotone convergence argument  shows that it suffices to prove the approximation results for measures $\nu$ which are compactly supported in $X_0$ since the assertions \eqref{s-evi} and \eqref{evi-disc} are void if $\nu$ is not supported in $X_0$.

Thus, let us now assume that $\nu$ has compact support $K$ in $X_0$ and that $V\ge0$. 
Fix $\epsilon>0$ such that $K':=\overline B_\epsilon(K)$ is compact in $X_0$.
Let 
$\overline\nu_t$, $t>0$, be the gradient flow  in $\mathcal{P}_2(X')$ for the entropy $\mathrm{Ent}(.|m')$ starting in $\nu$   or, in other words, the heat distribution on $X'$ after time $t$ for the initial distribution $\nu$. 
Put 
$\nu_t(.)=\frac1{\alpha_t}\overline\nu_t(K'\cap . )$  where 
$\alpha_t=\overline\nu_t(K')$.
Then continuity of $t\mapsto \overline\nu_t$ in $\mathcal{P}_2(X')$ implies 
$\alpha_t\to1$ as $t\to0$.
Thus
\begin{eqnarray*}
\limsup_{t\to0} 
\int Vd\nu_t=\limsup_{t\to0}\frac1{\alpha_t} \int_{K'} Vd\overline\nu_t\le\int_{K'} Vd\nu=\int Vd\nu
\end{eqnarray*}
since $\overline\nu_t\to\nu$ weakly and since $V$ is bounded and continuous on $K'$.
Moreover, obviously
$\mathrm{Ent}(\nu_t|m')<\infty$ for all $t>0$ and $W_2(\nu_t,\nu)\to0$ as $t\to0$.

{\bf (v)}
Assume  that the accumulation points of $(\mu^n_t)_{n\in\N}$ are not unique, say $\tilde\mu_t$ is another accumulation point (again for simplicity along one fixed subsequence for all rational $t$).
Applying the previous estimate  \eqref{s-evi} twice (first with $\tilde\mu_s$  in the place of $\nu$, then with 
$\mu_{t}$ in the place of $\nu$ and $\tilde\mu_s,  \tilde\mu_t$ in the place of $\mu_s,\mu_t$)
yields
\begin{eqnarray*}
\frac12  W^2_2(\tilde\mu_t,\mu_t)e^{2\kappa t}+\frac12  W^2_2(\tilde\mu_s,\mu_t)[e^{2\kappa t}-e^{2\kappa s}]
-\frac12  W^2_2(\mu_s,\tilde\mu_s)e^{2\kappa s}
\le 
\frac1{2\kappa}(e^{2\kappa t}-e^{2\kappa s})[
 S(\tilde\mu_s)- S(\tilde\mu_t)]
\end{eqnarray*} 
and thus in the limit $t-s\to0$
\begin{equation*}
\frac{d}{dt} \frac12 W_2^2(\mu_t,\tilde\mu_t) +\kappa \, W_2^2(\mu_t,\tilde\mu_t)\le0 
\end{equation*}
for a.e.\ $t$. This finally implies $W_2(\mu_t,\mu_t')=0$ and hence proves the uniqueness of the accumulation points.

{\bf (vi)}
Now let us  consider two flows starting in different points, say $\mu_0,\mu'_0\in\mathcal{P}_2(X')$.
Again from 
\eqref{evi-disc}
we may deduce 
$\frac{d}{dt} \frac12 W_2^2(\mu_t,\mu'_t) +\kappa \, W_2^2(\mu_t,\mu'_t)\le0 
$
which immediately yields the $W_2$-expansion bound
\begin{equation}\label{W2-contr}W_2(\mu_t,\mu_t')\le e^{-\kappa t} W_2(\mu_0,\mu_0').\end{equation}
\end{proof}

\begin{lemma}
For each $\mu_0\in {\mathcal P}_2(X')$ the EVI$_\kappa$-flow for $S$ satisfies
$$\mathrm{Var}_2(\mu_t)\le e^{-2\kappa\, t} \, \mathrm{Var}_2(\mu_0)$$
for all $t>0$ where 
$\mathrm{Var}_2(\nu):=\int \int d^2(x,y)\, \nu(dx)\nu(dy)$.
\end{lemma}

\begin{proof}
The  EVI$_\kappa$-property (\ref{s-evi}) for  $(\mu_t)_{t\ge0}$ applied to $\nu=\delta_y$ states
\begin{eqnarray}
\frac12 \int d^2 (x,y)\, d\mu_t(x)e^{2\kappa t}-\frac12  
\int d^2 (x,y)\, d\mu_s(x)e^{2\kappa s}
\le 
\frac1{2\kappa}(e^{2\kappa t}-e^{2\kappa s})[
 V(y)- S(\mu_t)].
\end{eqnarray}
Integrating this w.r.t. $\mu_t(dy)$ or  $\mu_s(dy)$, resp., yields
 \begin{eqnarray}
\frac12 \mathrm{Var}_2(\mu_t)e^{2\kappa t}-\frac12  \int
\int d^2 (x,y)\, d\mu_s(x)d\mu_t(x)e^{2\kappa s}
\le 0
\end{eqnarray}
and
 \begin{eqnarray}
\frac12  \int
\int d^2 (x,y)\, d\mu_s(x)d\mu_t(x)e^{2\kappa t}-
\frac12 \mathrm{Var}_2(\mu_s)e^{2\kappa s}
\le 
\frac1{2\kappa}(e^{2\kappa t}-e^{2\kappa s})[
 S(\mu_s)- S(\mu_t)].
\end{eqnarray}
Adding up these two inequalities and letting $s\to t$ we obtain
 \begin{eqnarray*}
\frac{d}{dt} \frac12 \mathrm{Var}_2(\mu_t)e^{2\kappa t}\le0
\end{eqnarray*}
since $S(\mu_s)\to S(\mu_t)$. This proves the claim.
\end{proof}

\section*{Appendix 1}
Let us finally add some observations which are not directly related to our main result but which might be of independent interest.

Denote the unique EVI$_\kappa$-flow for $V$ starting in $x_0\in X'$ by $x_t=\Phi_t(x_0), t>0$. Then for each $t>0$ the map $\Phi_t:X'\to X'$ is Lipschitz continuous (with Lipschitz constant 
$e^{-\kappa t}$). The family of maps $\Phi_t, t>0$ is a semigroup w.r.t. composition:
$\Phi_t\circ \Phi_s=\Phi_{t+s}$.

\begin{corollary}
For each $\mu_0\in\mathcal{P}_2(X')$  the (unique) EVI$_\kappa$-flow for $S$ starting in $\mu_0$
is obtained via push forward as
 \begin{equation}\label{push}
\mu_t=(\Phi_t)_*\mu_0.
\end{equation}
That is, $\int f(x)d\mu_t(x)=\int f(\Phi_t(x)d\mu_0(x)$ for each bounded Borel function $f$ on $X$.
\end{corollary}

This is a straightforward consequence of the following lemma which states  that
the EVI$_\kappa$-flow for $S$ is $\sigma$-additive in $\mu$.

\begin{lemma}
Assume that $\mu_0=\sum_{i=1}^n \lambda_i \mu_0^i$ for some probability measures $\mu_0^i$ and positive numbers $\lambda_i$ with $\sum_i \lambda_i=1$ where $n\in\N$ or $n=\infty$. For each $i$, let $(\mu_t^i)_{t>0}$ be the EVI$_\kappa$-flow for $S$ starting in $\mu_0^i$. Then 
\begin{equation}\label{sum-of-evis}
\mu_t=  \sum_{i=1}^n \lambda_i \mu_t^i
\end{equation}
is the EVI$_\kappa$-flow starting in $\mu_0$.
\end{lemma}

\begin{proof} 
Let us define $\mu_t$ as in (\ref{sum-of-evis}) and for fixed $t$ let $q_t$ be an optimal coupling of $\mu_t$ and $\nu$. Observe that 
$\mu^i_t=f_t^i \mu_t$ for suitable bounded nonnegative functions $f_t^i$ on $X$.
Put
$dq_t^i(x,y)=f_t^i(x)dq_t(x,y)$ and $\nu_t^i(.)=q_t^i(X,.)$.
Then
$\nu$ and $q_t$ can be represented as
\begin{equation*}\label{sum-of-ots}
\nu=  \sum_{i=1}^n \lambda_i \nu^i,\quad
q_t=  \sum_{i=1}^n \lambda_i q_t^i
\end{equation*}
where for each $i$ the probability measure $q^i_t$ is an optimal coupling of $\mu_t^i$ and $\nu^i$
(since sub-couplings of optimal couplings are optimal).
Thus
\begin{equation}\label{sum-of-ws1}
W_2^2(\mu_t,\nu)=\sum_i \lambda_i W_2^2(\mu_t^i,\nu_t^i).
\end{equation}
On the other hand, for each $s>t$
\begin{equation}\label{sum-of-ws2}
W_2^2(\mu_s,\nu)\le\sum_i \lambda_i W_2^2(\mu_s^i,\nu_t^i).
\end{equation}
(One should not expect equality since in general the optimal decomposition of $\nu$ induced by the decomposition $\mu_t=  \sum_{i=1}^n \lambda_i \mu_t^i$ will not be optimal for the coupling with
$\mu_s=  \sum_{i=1}^n \lambda_i \mu_s^i$. )
The EVI$_\kappa$-property for $(\mu_s^i)_{s>t}$ (with observation point $\nu_t^i$) implies
\begin{equation}\label{sum-of-ws3}
\lim_{s\searrow t}\frac1{2(s-t)}\Big[W_2^2(\mu_s^i,\nu_t^i)-W_2^2(\mu_t^i,\nu_t^i)\Big]
+\kappa W_2^2(\mu_t^i,\nu_t^i)  \le S(\nu_t^i)- S(\mu_t^i).
\end{equation}
Adding up the last three inequalities yields
\begin{eqnarray*}\label{sum-of-ws4}
\lefteqn{\limsup_{s\searrow t}\frac1{2(s-t)}\Big[W_2^2(\mu_s,\nu)-W_2^2(\mu_t,\nu)\Big]
+\kappa W_2^2(\mu_t,\nu) }\\
&\le &\sum_i \lambda_i \Big[
\lim_{s\searrow t}\frac1{2(s-t)}\Big(W_2^2(\mu_s^i,\nu_t^i)-W_2^2(\mu_t^i,\nu_t^i)\Big)
+\kappa W_2^2(\mu_t^i,\nu_t^i) \Big]
\\
&\le&
\sum_i \lambda_i 
\Big[ S(\nu_t^i)- S(\mu_t^i)\Big]\\
&=&
\Big[ S(\nu_t)- S(\mu_t)\Big].
\end{eqnarray*}
For the last equality, we made use of the additivity of $S$.
\end{proof}

\section*{Appendix 2}

Finally, let  us consider EVI-gradient flows in a more general setting.
Let $(X,d)$ be an arbitrary complete geodesic space and let $V:X\to(-\infty,+\infty]$ be a continuous function, finite on a dense subset $X_0$ of $X$, and let
a family of maps $\Phi_t: X\to X$ be given such that for each $x\in X$ the curve
$(\Phi_t(x))_{t\ge0}$ is the unique EVI$_\kappa$-gradient flow for $V$ (for some fixed $\kappa\in\R$).
This gradient flow comes with  the `natural' parametrization 
\begin{equation*}
\frac{d}{dt}V(\Phi_t)=-\left|\nabla^- V(\Phi_t)\right|^2=-|\dot \Phi_t|^2
\end{equation*}
We will now have a closer look at a reparametization $(\Psi_a(x))_{a\in\R}$ of this flow with
$\frac{\partial}{\partial a}V(\Psi_a)=1$.
For $a\in\R$ and $x\in X$ we put
\begin{equation*}
T_a(x)=\inf\left\{t\ge0: V(\Phi_t(x))\le a\right\}
\end{equation*}
(with $\inf\emptyset:=\infty$)
and if $T_a(x)<\infty$
\begin{equation*}
\Psi_a(x)=\Phi_{T_a(x)}(x).
\end{equation*}
Continuity of $x\mapsto V(x)$ and $t\mapsto \Phi_t(x)$ for $t>0$ imply that $V(\Psi_a(x))=a$ provided
$0<T_a(x)<\infty$.

\begin{proposition*} For all $x,x'\in X$ and all $a\in\R$ with $T_a(x), T_a(x')<\infty$
\begin{equation}
d\Big(\Psi_a(x),\Psi_a(x')\Big)\le \exp\left(-\kappa\frac{ T_a(x)+T_a(x')}2\right)\cdot d\big(x,x'\big).
\end{equation}
\end{proposition*}

\begin{proof}
Given $x,x'\in X$ and $a\in\R$ put
$T=T_a(x)\wedge T_a(x')$ and $S=|T_a(x)-T_a(x')|$. Assume without restriction that $T_a(x)\le T_a(x')$.
Then $T_a(x)=T$, $T_a(x')=T+S$ and $\Psi_a(x)=\Phi_T(x)$, $\Psi_a(x')=\Phi_{T+S}(x')$.
Applying the basic contraction property \eqref{contract} with $t=T$ yields
\begin{equation}\label{first}
d(\Phi_T(x),\Phi_T(x'))\le e^{-\kappa\, T} d(x,x').
\end{equation}
Next we apply \eqref{evi} to the flow $(\Phi_{T+t}(x'))_{t>0}$  with observation point $z=\Phi_T(x)$.
It states that
\begin{equation*}
\frac{d}{dt} \frac12 d^2(\Phi_T(x),\Phi_{T+t}(x'))
 +\kappa \, d^2(\Phi_T(x),\Phi_{T+t}(x'))\le V(\Phi_T(x))-V(\Phi_{T+t}(x'))\le0
\end{equation*}
 for $t\in [0,S]$
which immediately yields
\begin{equation*}
d(\Phi_T(x),\Phi_{T+S}(x'))\le e^{-\kappa\, S/2} d(\Phi_T(x),\Phi_T(x')).
\end{equation*}
Together with \eqref{first} we thus obtain
\begin{equation*}\label{second}
d(\Phi_T(x),\Phi_{T+S}(x'))\le e^{-\kappa\, (T+S/2)} d(x,x')
=\exp\left(-\frac\kappa2( T_a(x)+T_a(x'))\right)\cdot d(x,x').
\end{equation*}
This is the claim.
\end{proof}

\medskip

{\it The author gratefully acknowledges valuable comments of Matthias Erbar, Nicola Gigli, Martin Kell, and Christian Ketterer on early drafts of this paper. }

\end{document}